\documentclass{amsart}
\usepackage{graphicx}
\usepackage{amsthm,amsmath,verbatim,amssymb}
\usepackage{caption}
\usepackage{color}
\newtheorem{lem}{Lemma}

\newtheorem{thm}{Theorem}

\newcommand{\CP}{\mathbb{CP}^1}
\newcommand{\ga}{\gamma}

\newcommand{\kahler}{K\"ahler }
\renewcommand{\d}{\partial}
\newcommand{\dbar}{\bar\partial}

\newcommand{\eh}{e^{-\frac{n\phi}2}}
\newcommand{\R}{\mathbb{R}}
\newcommand{\inv}{^{-1}}

\newcommand{\hn}{H^0(M,L^n)}
\newcommand{\hnp}{H_p^0(M,L^n)}
\newcommand{\hnq}{H_q^0(M,L^n)}
\newtheorem{rem}{Remark}

\newcommand{\e}{\mathbf{E}}
\newcommand{\De}{\Delta}

\newcommand{\C}{\mathbb{C}}

\begin{document}
\title[Conditional expectations]{Conditional expectations  of  random holomorphic fields on Riemann surfaces }
\author[Renjie Feng]{Renjie Feng}
\address{Beijing International Center for Mathematical Research, Peking University, Beijing, China}
\email{renjie@math.pku.edu.cn}

\date{\today}
   \maketitle
   \begin{abstract}

 We study two conditional expectations:  $K_n(z| p)$ of the expected density of critical points of Gaussian random holomorphic sections $s_n \in\hn$ of powers of a positive holomorphic line bundle $(L,h)$ over Riemann surfaces $(M,\omega)$ given that the random sections vanish at a point and $D_n(z|q)$ of the expected density of zeros given that the random sections has a fixed critical point.  The critical points are points $\nabla_{h^n} s_n=0$ where $\nabla_{h^n}$ is the smooth Chern connection of the Hermitian metric $h^n$.  The main result is that the rescaling conditional expectations $K_n(p+\frac u{\sqrt n}|p)$ and $D_n(q+\frac u{\sqrt n}|q)$ have  universal  limits $K_\infty(u|0)$ and $D_\infty(v|0)$ as the power of the line bundle tends to infinity. We will see that the short distance behaviors of these two conditional expectations are quite different: the behavior between critical points and the conditioning zero is neutral while there is a repulsion between zeros and the conditioning critical point. But the long distance behaviors of these two rescaling densities are the same.

   \end{abstract}

\section{Introduction}
For Gaussian random polynomials of degree $n$, we study the conditional expectation of  critical points given that the polynomials vanish at a point and the conditional expectation of zeros given that the polynomials have a fixed critical point. More generally, we consider
the conditional distribution $K_n(z| p)$  of critical points  of Gaussian random holomorphic sections
$s_n \in\hn$ of powers of a positive holomorphic line bundle $(L,h)$ over Riemann surface $(M,\omega)$  given that the random sections vanish at a point $p$ and the conditional expectation $D_n(z|q)$  of zeros  given that the Chern connection of the random sections vanish at a point $q$.  We will  apply a Kac-Rice type formula to derive
 $ K_n(z|p)$ and the probabilistic Poincar\'{e}-Lelong formula to derive $D_n(z|q)$, then we rescale them to prove that  $K_n(p+\frac u{\sqrt n}|p)$ and $D_n(q+\frac u{\sqrt n}|q)$ have  universal limits as $n$ tends to infinity.
%which reflects some local bahaviors between critical points and conditioning zero.

The motivation of this paper is to study the local behavior between the critical points and zeros of random holomorphic fields. The famous Gauss-Lucas Theorem states that the holomorphic critical points of any polynomial are contained in the convex hull of its zeros.  This implies that some non-trivial correlations between zeros and critical points of
 random polynomials may exist.  This problem has been studied recently in \cite{Hann1, Hann2} for Gaussian random $SU(2)$ polynomials where
  a two-point correlation function between zeros and critical points is derived. It is also proved that on
the $n^{-\frac 12}$ length-scaled,  zeros and  critical points appear
in rigid pairs. It seems that the similar results hold for holomorphic sections of line bundles over Riemann surfaces. In this article, we study the analogous problems.  Instead of two-point correlation between zeros and critical points, we study the
 conditional expectations of critical points and zeros. Our essential setting on Riemann surfaces is that the critical points are defined as zeros of the derivative of the smooth Chern connection $\nabla_{h^n}$ instead of
 the meromorphic connection (or locally, the holomorphic derivative $\frac{\partial}{\partial z}$
on a coordinate patch $\C$ of Riemann surfaces).  %The local behavior of critical points and conditioning zeros is resembled by the rescaling limit $K_n(p+\frac{p}{\sqrt n}|p)$.

% More precisely, suppose the $SU(2)$ polynomial is conditioned to
%have a zero at $z_0\in \C\setminus \{0\}$; for any $\epsilon\in(0,\frac12)$
%there will be a unique critical
%point in the annulus  $\{z\in\C:\,\, n^{-1-\epsilon}<|z-z_0|<n^{-1+\epsilon}\}$.

\subsection{Results on critical points}
To state our results, we need to recall the definition of Gaussian random holomorphic sections of a line bundle (see \S \ref{bg}).
 We let $(L,h)\to (M,\omega)$ be a positive Hermitian holomorphic line bundle
over a Riemann surface with the \kahler form $\omega=\frac{\sqrt{-1}}{2}\Theta_h$, where $\Theta_h$ is the curvature of $h$. We denote
$H^0(M,L^n)$ as the space of global holomorphic sections of $n$-th tensor power of $L$.
A special case is when $M=\CP$ and $L=\mathcal O(1)$ the hyperplane line bundle, $H^0(\CP, \mathcal O(n))$ is the space of homogeneous polynomials of degree $n$. The Hermitian metric $h$ will induce an inner product on $H^0(M, L^n)$
and thus induces a Gaussian measure $d\gamma_{d_n}$ on $H^0(M, L^n)$, where $d_n$ is the dimension of  $H^0(M, L^n)$.

 We define $K_n(z|p)$ the conditional expectation of critical points as a $(1,1)$-current
\begin{equation}\label{firstdef}\int_M\psi  K_n(z|p) =\e_{(H^0(M, L^n),d\gamma_{d_n})}\left(\sum_{z: \nabla_{h^n}s_n=0 } \psi(z) |s_n(p)=0\right),\end{equation}
for any test function $\psi\in C_0^\infty(M)$ where $\nabla_{h^n}$ is the Chern connection. In \S\ref{ce}, we will rewrite the right hand side as an expectation taken in the probability space $(H^0_p(M, L^n), d\gamma^p_{d_n-1})$ with respect to the conditional Gaussian measure  $d\gamma^p_{d_n-1}$ (see \S \ref{ce}).

%The idea to study $K_n(z|p)$ is to study the conditional Bergman kernel $\Pi_n^p(z,w)$ and its rescaling limit.
 % Let $H^p_0\subset H^0(M, L^n)$ be the space of holomorphic sections vanishing at $p$.
 %Let $\{s_1^p, ..., s_{d_n-1}^p\}$
 %be an orthonormal basis of  $H^p_0$ under the inner product \ref{innera2}.
 %By the properties of reproducing kernel $\Pi_n(z,w)$, we can show that $\{s_1^p, ..., s_{d_n-1}^p, %\Phi_n^p\}$
% will be an orthonormal basis for $H^0(M, L^n)$ where $\Phi_n^p(z)=\frac{\Pi_n(z,p)}{\|\Pi_n(p,p)\|%_{h^n}^{1/2}}$ where $\Pi_n$ is the Bergman kernel.
%Then the conditional Szeg\"o projection $\Pi_n^p$ is given by
 %$$\Pi_n^p(z,w)=\sum s_j^p(z)\otimes \bar s^p_j(w)=\Pi_n(z,w)-\Phi_n^p(z)\otimes %\bar\Phi_n^p(w)$$
   In order to get the conditional distribution of the critical points, we need to apply the generalized Kac-Rice formula for complex manifolds \cite{BSZ1, DSZ1,DSZ2}. Our first result
 is the following Kac-Rice type formula for the global $(1,1)$-current of $K_n(z|p)$.
\begin{thm}\label{thm1}
Let $(L, h)\to(M,\omega)$ be a positive Hermitian holomorphic line bundle over a compact Riemann surface with the \kahler form $\omega=\frac{\sqrt{-1}}{2}\Theta_h$, let $(H^0(M, L^n),d\gamma_{d_n})$
be a complex Gaussian ensemble defined in \S \ref{Gaussianva}. Then the conditional expectation of the empirical measure of critical points given that the random sections vanish at $p$ is
$$K_n(z|p)=   \left(\int_{\C^2} \frac 1{\pi^{3} }\frac 1{ A_n \det \Lambda_n}\exp\left\{-\left\langle \begin{pmatrix} \xi  \\y\end{pmatrix}, \Lambda_n^{-1}\begin{pmatrix} \bar \xi \\ \bar y \end{pmatrix}\right \rangle\right\}\left| |\xi|^2-n^2   |y|^2 \right|d\ell_yd\ell_\xi\right)\omega(z),$$
where $d\ell_y$ and $d\ell_\xi$ are Lebesgue measures on $\C$ and
$$\Lambda_n=C_n- A_n^{-1}B_n^* B_n$$
where $$A_n=\d_{  z}\d _{\bar w}\Pi_n^p(z,w)|_{z=w},$$
    $$B_n=(\d_{  z}\d^2_{\bar w}\Pi_n^p(z,w)|_{z=w} ,\d_{  z}\Pi_n^p(z,w)|_{z=w}),$$
   and
   $$C_n=\begin{pmatrix}\d_{  z}^2\d_{\bar w}^2\Pi_n^p(z,w)|_{z=w}& \d_{  z}^2\Pi_n^p(z,w)|_{z=w}\\ \d^2_{\bar w}\Pi_n^p(z,w)|_{z=w}&\Pi_n^p(z,z) \end{pmatrix},$$
where
$$\Pi_n^p(z,w)=\Pi_n(z,w)-\frac{\Pi_n(z,p)\overline{\Pi_n(w,p)}}{\Pi_n(p,p)}, $$
 where $\Pi_n(z,w)$ is the Bergman kernel which is the projection of the $L^2$ integral sections to the holomorphic sections (see \S \ref{djdjdjd}).
\end{thm}

%As a consequence of Theorem \ref{thm1},  we can get a global result on the distribution of the
%critical points given sections vanishing at a point.
 %If we apply the asymptotic expansions of the  conditional Bergman kernel  $\Pi_n^p(z,w)$  and its %derivatives up to order $4$ to Theorem \ref{thm2},  we can derive the expected density of the %normalized empirical measure of the conditional critical points,
 %\begin{thm}\label{thm2}
 %$$\lim_{n\to\infty} \frac1n K_n(z|p)= $$

 %\end{thm}

We rescale the global expression of $K_n(z|p)$ to get the following local behavior between critical points and the conditioning zero,

\begin{thm}\label{main}
The rescaling limit of the $(1,1)$-current of the conditional expectation has a pointwise universal limit,
$$K_\infty(u|0):=\lim_{n\to\infty} K_n(p+\frac{u}{\sqrt n}|p)=\frac{1}{\pi a_\infty^2}\frac{(\lambda_1^\infty)^2+(\lambda_2^\infty)^2}{|\lambda_1^\infty|+|\lambda_2^\infty|}\frac{\sqrt{-1}}{2}du\wedge d\bar u$$
where
$$a_\infty=1+|u|^2, \,\,\lambda_1^\infty=  2+2|u|^2+|u|^4,\,\,\lambda_2^\infty=  -1+|u|^2e^{-|u|^2}+e^{-|u|^2} .$$
\end{thm}
We will first prove this result for the special case of Gaussian random $SU(2)$ polynomials in \S \ref{su2} and then prove the general cases in \S \ref{profooo}. To prove this result, we need  the estimates of the rescaling limits of the Bergman kernel $\Pi_n(p+\frac z{\sqrt n},p+ \frac w {\sqrt n})$ and its derivatives up to order $4$.

\subsection{Results on  zeros}
%We can also study the distribution of zeros with a conditioning critical point.
The conditional expectation $D_n(z|q)$ of zeros of Gaussian random holomorphic sections with a fixed critical point is defined similarly,
\begin{equation}\label{seconddef}\int_M\psi  D_n(z|q) =\e_{(H^0(M, L^n),d\gamma_{d_n})}\left(\sum_{z: s_n=0 } \psi(z) |\nabla_{h^n}s_n(q)=0\right)\end{equation}
for any test function $\psi$.% In \S \ref{ce}, we will also rewrite this expectation in \eqref{alternative1}.

In \S\ref{proofoooo}, we will apply the probabilistic Poincar\'{e}-Lelong formula to get,

\begin{thm}\label{thm2}
With the same assumptions in Theorem \ref{thm1},   the conditional expectation of the empirical measure of zeros given that the random sections have a critical point at $q$ is
$$D_n(z|q)=\frac{\sqrt{-1}}{2 \pi}\partial\bar\partial \log |\Pi_n^q(z,z)|,$$
where
$$\Pi_n^q(z,z)=\Pi_n(z,z)-\frac{|\partial_{\bar w}\Pi_n(z,w)|_{w=q}|^2}{(\partial_z\partial
_{\bar z}\Pi_n)(q,q)}.$$
\end{thm}

Furthermore, $D_n(z|q)$ admits the following universal limit,
%We rescale the above formula to get the following universal limit between zeros and the conditioning critical point,

\begin{thm}\label{main2}
$$D_\infty(v|0):=\lim_{n\to\infty}D_n(q+\frac{v}{\sqrt n}|q)=\frac{\sqrt{-1}}{2 \pi}\partial\bar\partial \log \left(e^{|v|^2}-|v|^2\right).$$
\end{thm}

% in the sense of average, we can always find critical points (or zeros) in the ball of the length scale $n^{-\frac{1}{2}}$
%around the conditioning zero (or critical point) as $n$ large enough.
 Theorem \ref{main} and Theorem \ref{main2} indicate that these two universal limits depend only on the distance $r=|u|$ or $|v|$ between the scaled point and the conditioning fixed point. The following graphs illustrate the growth of the density functions $K_\infty(u|0)$ and $D_\infty(v|0)$ (by discarding the Lebesgue measure).

\begin{center} \label{kg}
\includegraphics[scale=0.35]{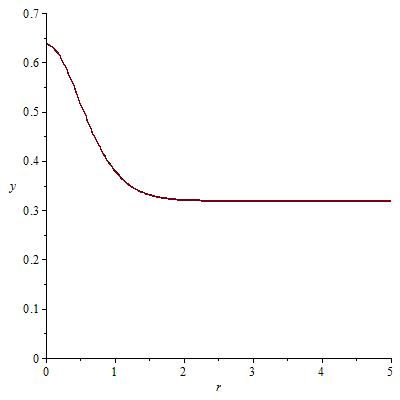}\par\text{Graph of  $K_\infty(u|0)$}
\end{center}

\begin{center} \label{kg}
\includegraphics[scale=0.35]{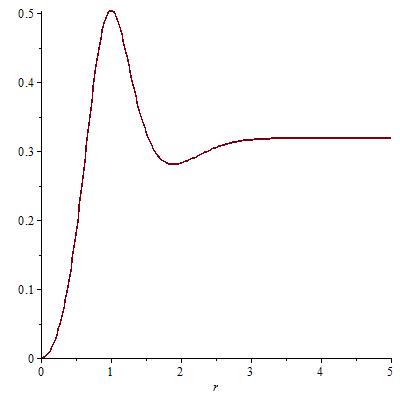}\par\text{Graph of  $D_\infty(v|0)$}
\end{center}

Theorem \ref{main} and Theorem \ref{main2} also determine some local behaviors
between critical points and  zeros. It's very surprising to see that the short distance behaviors of $K_\infty(u|0)$ and $D_\infty(v|0)$ are quite different,
\begin{equation}
\lim_{|u|\to 0}K_\infty(u|0)=\frac 2 \pi,\,\,\,\,\,\lim_{|v|\to 0}D_\infty(v|0)=0.
\end{equation}
But the long distance behaviors are the same,\begin{equation}
\lim_{|u|\to \infty}K_\infty(u|0)=\frac 1 {\pi},\,\,\,\,\,\lim_{|v|\to \infty}D_\infty(v|0)=\frac 1\pi .
\end{equation}

 Intuitively, the rescaling limit measures  the asymptotic probability of finding critical points/zeros  in the geodesic ball of length scale $n^{-\frac{1}{2}}$ centered at the conditioning point. Roughly speaking, the limit $D_\infty(v|0)$ tends to $0$ as $|v|\to 0$ indicates that given a critical point at $q$, it's unlikely to find
a zero near $q$, i.e., there is a `repulsion` between zeros and the conditioning critical point. Such behavior is quite different from that of the critical points given a zero: the limit of $K_\infty(u|0)$ tends to a constant as $|u|\to 0$ indicates that critical points
and the conditioning zero behave `neutrally` for the short distance. It should be very interesting to see the behaviors for the higher dimensions.
%In \cite{SZZh}, the authors studied the conditional distribution of zeros with a fixed zero at a given point for all dimensions, we refer to \S 6 for the comparison between the two-point correlation and the conditional distribution of zeros.

%and

% furthermore, the limit $D_\infty(v|0)\to \infty$ as $|v|\to \infty$ implies that it's more likely to find zeros at the infinity. For example, the Gaussian random
%$SU(2)$ polynomials are defined as the random holomorphic sections of the hyperplane line bundle over $\CP \cong S^2$ (see \S \ref{su2}),
%given a critical point at the north pole, it's more likely to find zeros at the south pole which is the infinity point of the Riemann sphere.

As a remark, the universal rescaling limits  $K_\infty(u|0)$ and  $D_\infty(v|0)$ are the rescaling conditional densities for the Bargmann-Fock case with the conditioning point at $z=0$, i.e., the corresponding conditional densities for the random holomorphic functions,
$$f(z)=\sum_{j=0}^\infty \frac {a_j}{\sqrt{j!}}z^j,$$
where $a_j$ are i.i.d. standard complex Gaussian random variables with mean $0$ and variance $1$. The reason of this can be tell from the proof: the covariance kernel of Gaussian random holomorphic sections is expressed by the Bergman kernel, while the rescaling limit of  the Bergman kernel on any polarized line bundle over \kahler manifolds is universal which is the rescaling limit of the Bergman kernel of the Bargmann-Fock space (see Remark \ref{bgj}).

\bigskip

\textbf{ Acknowledgement:} The author would like to thank Steve Zelditch for his generous support for so many years.
 He would like to thank Bernard Shiffman, Zuoqin Wang, Zhiqin Lu and Zhenan Wang for many helpful discussions.
Many thanks also go to Richard Wentworth, Robert Adler and Bo Guan.

%We choose a small ball $B_p(\frac{\epsilon}{\sqrt n})$ under the \kahler normal coordinate centered at $p$ and we take the test function as the %characteristic function $\chi_{B_p(\frac{\epsilon}{\sqrt n})}$ in the definition of the conditional expectation \eqref{firstdef}, we have

 %\begin{align*}&\e\left( \#\, \mbox{critical points in}\, B_p(\frac{\epsilon}{\sqrt n}) |s_n(p)=0 \right)\\ =&\int_{B_p(\frac{\epsilon}{\sqrt n})}K_n(z|p)\\
 %=&\int_{B_p(\epsilon) }K_n(p+\frac{z}{\sqrt n}|p)\\
 %\to& \int_{B_p(\epsilon) }K_\infty^p, \,\,\,\,\,\,as\,\,n\to\infty
%\end{align*}

%Note that the last integration is increasing to infinity as a function of $\epsilon$, i.e., there exists $\epsilon_0$ so that the integration is strictly greater than %$1$.

 \section{Background}\label{bg}
 In this section, we will review some basic concepts and notations on Gaussian random holomorphic sections of
 positive holomorphic line bundles over Riemann surfaces. We refer to \cite{BSZ1, GH} for more details.

\subsection{\kahler manifolds}
Let $(M,\omega)$ be a compact Riemann surface (which is a \kahler manifold) with the \kahler form \begin{equation}\omega=\frac{\d^2\phi}{\d z \d\bar z}\frac{\sqrt{-1}}{2}dz\wedge d\bar z,\end{equation} where $\phi$ is the local \kahler potential in a local coordinate patch $U\subset M$. Let $(L,h)\to (M,h)$ be a positive holomorphic line bundle such that the curvature of the Hermitian metric $h$
 \begin{equation}\Theta_h=-\frac{\d^2 \log h }{\d z \d\bar z}\frac{\sqrt{-1}}{2}dz\wedge d\bar z\end{equation}
is a positive $(1,1)$ form \cite{GH}. Let $e$ be a local non-vanishing holomorphic section of $L$ over $U \subset M$ such that locally $L|_U\cong U\times\C$ and the pointwise $h$-norm of $e$ is $|e|_{h} = h(e, e)^{
1/2}$. Throughout the article, we further assume that the line bundle is polarized, i.e., $\Theta_h=\omega$ or equivalently $|e|^2_h=h(e,e)=e^{-\phi}$.

We denote by $H^0(M,L^n)$ the space of global holomorphic sections of $n$-th tensor power of $L$. Locally, we can write the global holomorphic section of $L^n$ as $s_n=f_ne^{\otimes n}$ where $f_n$ is a holomorphic function on $U$.
We denote the dimension of $H^0(M,L^n)$ by $d_n$. The Hermitian
metric $h$ induces a Hermitian metric $h
^n$ on $L^n$ given by $|e^{ \otimes n}|_{h^n}=|e|_h^n$,  i.e., $|s_n|^2_{h^n}=|f_n|^2h^n(e^{\otimes n},e^{\otimes n})=|f_n|^2e^{-n\phi}$.

We decompose the smooth Chern connection  $\nabla_{h^n}=\nabla'_{h^n}+\nabla''_{h^n}$ of the Hermitian line bundle $(L^n,h^n)$  into holomorphic and antiholomorphic parts where in the local coordinate $\nabla'_{h^n}=d_z+n\d\log h$ and $\nabla''_{h^n}=d_{\bar z}$. For the polarized line bundles, given a global holomorphic section $s_n=f_ne^{\otimes n}$, we have the following formula \cite{GH}
 \begin{equation}\label{dddddk}\nabla_{h^n}s_n=(\frac{\partial f_n}{\partial z}-n\frac{\partial \phi}{\partial z}f_n) e^{\otimes n}\otimes dz. \end{equation}

We can define an inner product on $H^0(M,L^n)$ as
\begin{equation}\label{innera}\langle s_{n,1}, s_{n,2}\rangle_{h^n}=\int_M h^n( s_{n,1}, s_{n,2})\omega.\end{equation}
% Throughout the article we assume the volume is normalized as $\int_M \omega=1$.

Under the local coordinate it reads,

\begin{equation}\label{innera2}\langle s_{n,1}, s_{n,2}\rangle_{h^n}=\int_M f_{n,1}, \overline{f_{n,2}}h^n(e^{\otimes n},e^{\otimes n}) \omega=\int_M f_{n,1}, \overline{f_{n,2}} e^{-n\phi}\omega,\end{equation}
where
$ s_{n,1}=f_{n,1}e^{\otimes n}$ and $s_{n,2}= f_{n,2}e^{\otimes n}$.

\subsection{Gaussian random fields}\label{Gaussianva}

Let's recall that  a complex  Gaussian measure on $\C^k$ is a measure of the
form
\begin{equation}\label{cxgauss}d\ga_\De= \frac{e^{-\langle\De\inv
z,\bar z\rangle}}{\pi^{k}{\det\De}} d\ell^k_z\,,\end{equation}
where $d\ell_z^k$ denotes Lebesgue measure on $\C^k$ and $\De$ is a
positive definite Hermitian $k\times k$ matrix.  The  matrix
$\De$ is the covariance matrix.% of $\ga_\De$:
%\begin{equation}\label{covar}
%\langle
%z_\al\bar z_\be\rangle_{\ga_\De} =\De_{\al\be},\quad\quad 1\leq \al,\be\leq
%k\;.\end{equation}

The inner product \eqref{innera} induces a complex Gaussian probability measure $d\gamma_{d_n}$ on the space $\hn$ as,
\begin{equation}\label{standardgauss}d\gamma_{d_n}(s_n)=\frac{e^{-|a|^2}}{\pi^{d_n}}da,\,\,\,\,\, s_n=\sum_{j=1}^{d_n}a_j s_{n,j},\end{equation}
where $\{s_{n,1},..., s_{n,d_n}\}$
 is an orthonormal basis for $\hn$ and $\{a_1,..., a_{d_n}\}$ are i.i.d. standard complex Gaussian random
 variables with mean $0$ and variance $1$. %Throughout the article, we will use to denote such standard complex Gaussian random fields on Riemann surfaces.

   \section{Conditional expectation}\label{ce}
   In this section, we will rewrite the conditional expectations of
   the empirical measures  in \eqref{firstdef}\eqref{seconddef} by defining the conditional Gaussian measure, then we will derive the covariance kernels for the conditional Gaussian processes.
\subsection{Conditional Gaussian measure}
   Let $d\gamma$ be a complex Gaussian measure on a finite dimensional
   complex vector space $V$, let $ W$ be a vector subspace of $V$. We define the conditional Gaussian measure $d\gamma_W$ on $W$ to be the restriction of $d\gamma$ on $W$ associated with
   the Hermitian inner product on $W$ induced by the inner product on $V$ \cite{SZZh}.

    %Proposition 3.6 in \cite{SZZh} proves the following statement: Let $T:\C^n\to V$ be a linear map onto a complex
   %vector space $V$, then \begin{equation}\label{ffff}\e_{\gamma_n}(X|T=0)=\e_{\gamma_{ker T}} (X')\end{equation}
   %where $X'$ is the restriction of $X$ to $ker T$ and $\gamma_{ker T}$ is the conditional Gaussian measure on the kernel of $T$.

    This definition can be understood as follows. For simplicity we let $ W$ be a vector subspace of $V$ of codimension $1$. We let $\{v_1,...,v_{m}\}$ be an orthonormal basis of $V$ and $d\gamma$ be a Gaussian measure induced by the inner product,

  $$d\gamma(v)=\frac{e^{-|a|^2}}{\pi^{m}}da,\,\,\,\, v=\sum_{j=1}^m a_j v_j,$$
   where $a_j$ are i.i.d. standard Gaussian random variables.

   Let $\{w_1,...,w_{m-1}\}$ be an orthonormal basis of $W$ where the inner product is induced by that of $V$. Then we extend it to an orthonormal basis  $\{w_1,...,w_{m-1}, w_m\}$ of $V$. There exists a unitary group $U=(u_{ij})$ such that $v_j=\sum_{i=1}^m u_{ji}w_i$. Then we can express $v=\sum_{j=1}^m a_j (\sum_{i=1}^m u_{ji}w_i)=\sum_{i=1}^m (\sum_{j=1}^ma_j u_{ji})w_i$. By the definition of the conditional expectation, the conditional Gaussian measure $d\gamma_W$ induced by $d\gamma$ is actually induced by the conditional Gaussian process  $$w=\sum_{i=1}^{m-1} (\sum_{j=1}^ma_j u_{ji})w_i,$$
which is the projection of $v$ to $W$.
The conditional expectation taken with respect to $d\gamma(v)$ is actually taken with respect to $d\gamma_W(w)$.
 %this explains the formula \eqref{ffff}.
  Furthermore, we can compute the covariance kernel of the conditional Gaussian process. The covariance kernel for the Gaussian process $v$ is
$$C_V(v(x), v(y))=\e_{d\gamma} (v(x)\bar v(y))=\sum_{j=1}^m v_j(x)\bar v_j(y)=\sum_{j=1}^m w_j(x)\bar w_j(y),$$
and the covariance kernel for the conditional Gaussian process is
\begin{align*}C_W(w(x), w(y))&=\e_{d\gamma_W} (w(x)\bar w(y))\\ &=\e_{d\gamma}\left([\sum_{i=1}^{m-1} (\sum_{j=1}^ma_j u_{ji})w_i][\sum_{i=1}^{m-1} (\sum_{j=1}^m\bar a_j \bar u_{ji})\bar w_i]\right)\\&  =\sum_{j=1}^{m-1} w_j(x)\bar w_j(y).\end{align*}
Hence, we have the following crucial relation,
\begin{equation}\label{relation}
C_W(w(x), w(y))=C_V(v(x), v(y))-w_m(x)\bar w_m(y).
\end{equation}

\subsection{Conditional densities}
   Now we can define the expected density of critical points of Gaussian random sections
   given that the random sections vanish at a point $p$ by the conditional Gaussian measure.
 We need the following bundle-value map\begin{equation}T: \hn \to L_p^n,\,\,\,s_n\to s_n(p).\end{equation}
   We define the kernel of $T$ as $\hnp\subset \hn$ which is a subspace of codimension $1$.

   We denote $C_{s_n}$ as the empirical measure of critical points of sections,
   \begin{equation}C_{s_n}=\{z\in M: \nabla_{h^n}s_n(z)=0\}.\end{equation}
   By definition of conditional Gaussian measure, we have the following relation, \begin{equation}\e_{\left(\hn,d\gamma_{d_n}\right)} (C_{s_n}| s_n(p)=0)=\e_{\left(\hnp, d\gamma^p_{d_n-1}\right)}(C_{s_n}),\end{equation}
  in the sense of distribution,
   \begin{equation}\label{rewritet}\langle \e_{\left(\hn,d\gamma_{d_n}\right)} (C_{s_n}| s_n(p)=0),\psi\rangle=
\e_{\left(\hnp, d\gamma^p_{d_n-1}\right)}\langle C_{s_n},\psi\rangle,
%=\e_{\left(\hnp, d\gamma_{d_n-1}\right)}\left(\sum_{z\in Z_{s_n}}\phi(z)\right)
    \end{equation}        for any test function $\psi\in C_0^\infty(M)$, where $d\gamma^p_{d_n-1}$ is the conditional Gaussian measure
    which is the restriction of $d\gamma_{d_n}$ on $\hnp$.
    %, but note it is not the standard Gaussian as defined in subsection \ref{Gaussianva}.

 We denote $K_n(z|p)$ as the conditional expectation $\e_{\left(\hn,d\gamma_{d_n}\right)} (C_{s_n}| s_n(p)=0)$, then $K_n(z|p)$ is a $(1,1)$-current \cite{GH} and we can rewrite \eqref{rewritet} as \begin{equation}\label{alternative}\int_M \psi K_n(z|p)=  \e_{\left(\hnp, d\gamma^p_{d_n-1}\right)}\left(\sum_{z\in C_{s_n}}\psi(z)\right). \end{equation}

 We need the following bundle-value linear map to define the conditional expectation of zeros given that the random sections have a critical point at $q$,
 \begin{equation}\label{linearmap2}K:  \hn\to (L^n\otimes T^{*'}M)_q,\,\,\,s_n\to \nabla_{h^n}s_n(q),
 \end{equation}
where $\nabla_{h^n}$ is the Chern connection.  We denote the kernel of this linear map  as $\hnq$. By the Kodaira embedding, $\hnq$ is a subspace of $\hn$ of codimension $1$. \footnote{We thank Prof. Zhiqin Lu for clarifying this.}

We denote   \begin{equation}Z_{s_n}=\{z\in M: s_n(z)=0\},\end{equation} and denote $$D_n(z|p)=:\e_{\left(\hn,d\gamma_{d_n}\right)} (Z_{s_n}| \nabla_{h^n}s_n(p)=0).$$ Then similarly, we have

\begin{equation}\label{alternative1}\int_M \psi D_n(z|q)=  \e_{\left(\hnq, d\gamma^q_{d_n-1}\right)}\left(\sum_{z\in Z_{s_n}}\psi(z)\right), \end{equation}
where $d\gamma^q_{d_n-1}$ is the conditional Gaussian measure on the subspace $\hnq$.
%,  we can find a basis $\{s_{n,1},\cdots, s_{n,{d_n-1}}\}$ of $\hnq$ with $\nabla_{h^n}s_{n,j}(q)=0, j=1,..., d_n-1$ and extend to a basis $\{s_{n,1},\cdots, %s_{n,{d_n-1}}, s_{n,{d_n}}\}$ of $\hn$ with $\nabla_{h^n}s_{n,d^n}(q)\neq 0$, the existence of such basis is guaranteed

  \section{Proof of Theorem \ref{thm1}}\label{thm122}

In this section, we will derive a Kac-Rice type formula for the global expression of the conditional expectation $K_n(z|p)$.
The formula may be derived from \cite{AT, BSZ1, DSZ1, DSZ2} but we take advantage of some simplifications to speed up the proof.
%We then apply this general Kac-Rice formula to get the global expression of $K_n(z|p)$; we will further simplify  the expression of $K_n(z|p)$.

\subsection{Kac-Rice formula}
%Let's derive a Kac-Rice formula for the global distribution of critical points given a zero.
 In the local coordinate $U\cong\C$ and a local trivialization of $L$, we write the conditional Gaussian random sections with a zero at $p$ as $s^p_n=f^p_ne^{\otimes n}.$
 We will prove the following,

\begin{lem}\label{joint1} The (1,1)-current of the conditional expectation of critical points of sections
with a conditioning zero at $p$ with respect to the Gaussian measure $d\gamma_{d_n}$ is
\begin{equation}K_n(z|p)=   \left(\int_{\C^2} p^n_{z}(y,0,\xi)\left| |\xi|^2-n^2   |y|^2 \right|d\ell_yd\ell_\xi\right)\omega,\end{equation}
where $p^n_{z}(y,s,\xi)$ is the joint density of the conditional Gaussian processes $(f_n^p,\frac{\d f_n^p}{\d z},\frac{\d ^2 f_n^p}{\d z^2})$; $d\ell_y$ and $d\ell_\xi$ are Lebesgue measures on $\C$.
\end{lem}
\begin{proof}
The strategy to get this formula is to find the local expression in a coordinate path $U\cong \mathbb C$, then turn it to be global.

We denote
\begin{equation}\Omega_p=\{z\in \C:\,\,(\frac{\d f^p_n}{\d z}-n\frac{\d\phi}{\d z} f^p_n)e^{-\frac{n\phi}2}=0\},\end{equation}
  then $\Omega_p$ is the same as the set of critical points $\{z\in \C:\, \nabla_{h^n} s^p_n=0\}$ which is $\{z\in \C:\,\frac{\d f^p_n}{\d z}-n\frac{\d\phi}{\d z} f^p_n =0\}$
	on the local coordinate patch (recall \eqref{dddddk}).

We first introduce some notations: \begin{equation}\label{dddddj}p_n=f^p_n\eh ,\,\,\, q_n=(\frac{\d f^p_n}{\d z}-n\frac{\d\phi}{\d z} f^p_n)\eh,\,\,\,r_n=\frac{\d ^2 f^p_n}{\d z^2}\eh,\end{equation} then
$p_n$, $q_n$ and $r_n$ are all complex Gaussian random variables.

By definition of the delta function, for any test functions $\psi\in C_0^\infty(\C)$ we have,
\begin{align*}&\langle \sum_{z\in \Omega_p}\delta_{z},\psi\rangle\\
=& \sum_{z:\,\,\, q_n(z)=0}\psi(z)\\=& \int_{\C} \delta_0(q_n)\psi(z) \frac{\sqrt{-1}}{2}dq_n\wedge d\bar{q}_n\\=&
\int_{\C} \delta_0(q_n)\psi(z)\left||\frac{\d q_n}{\d z}|^2-|\frac{\partial q_n}{\d\bar z}|^2\right|d\ell_z,
\end{align*} where $d\ell_z $ is the Lebesgue measure on $\C$.

By direct computations, we have,
\begin{align*}\frac{\d q_n}{\d z}&=(\d^2f^p_n-n\d\phi \d f^p_n-n\d^2\phi f^p_n)\eh-\frac n2 \d\phi q_n \\
&=r_n-n\d\phi q_n -n^2 (\d\phi)^2  p_n- n\d^2\phi p_n-\frac n2\d\phi q_n
\end{align*}
and
$$\frac{\d q_n}{\d \bar z}=-n\d\dbar \phi p_n-\frac n 2 \dbar\phi q_n.  $$
%where we use the notations $\d f=\frac{\d f}{\d z}$ and $\dbar f=\frac{\d f}{\d  \bar z}$ throughout the article .

By taking expectation on both sides, we have locally,
\begin{align*}&\e\langle  \sum_{z\in \Omega_p}\delta_{z},\psi\rangle \\
=&
\e\int_{\C} \delta_0(q_n)\psi(z)\left||\frac{\d q_n}{\d z}|^2-|\frac{\partial q_n}{\d\bar z}|^2\right|d\ell_z\\
=& \int_{\C^3}  \psi(z) p^n_z(y,0,\xi) \left||\xi-n^2 (\d\phi)^2 y-n\d^2\phi y|^2-n^2|\d\dbar \phi|^2 |y|^2 \right|
d\ell_\xi   d\ell_yd\ell_z
\end{align*}
where $p^n_z(y,s,\xi)$ is the joint probability of the Gaussian random field $(p_n, q_n, r_n)$, $d\ell_\xi$ and $d\ell_y$ are Lebesgue measures on $\C$.
Thus the conditional density is  locally given by the (1,1)-current,
\begin{equation}\label{d1}
\left(\int_{\C^2}   p^n_z(y,0,\xi) \left||\xi-n^2 (\d\phi)^2 y-n\d^2\phi y|^2-n^2|\d\dbar \phi|^2 |y|^2 \right|
d\ell_\xi   d\ell_y\right)d\ell_z
\end{equation} on the local coordinate patch $U \cong \C$ of   Riemann surfaces.

Now we need to get the global expression for the conditional density. Since the conditional expectation is a (1,1)-current globally defined on the Riemann surface (which is also independent of the local coordinate and the local frame),
it's sufficient to find the formula when we freeze at a point and the formula will turn out to be global.
For this purpose, given a complex $m$-dimensional \kahler manifold $(L,h)\to(M,\omega)$, we freeze at a point $z_0$ as the origin of the coordinate patch and choose the \kahler normal coordinate $\{z_j\}$ as well as an adapted frame $e_L$ of the line bundle $L$ around $z_0$.  It is well-known that in terms of \kahler normal coordinates $\{z_j\}$,
the K\"ahler potential $\phi$ has the following expansion in the neighborhood of the origin $z_0$, \begin{equation}\label{kahler}\phi(z,\bar z)= \|z\|^2 -\frac 14 \sum R_{j\bar
kp\bar q}(z_0)z_j\bar z_{\bar k} z_p\bar z_{\bar q} +
O(\|z\|^5)\;.\end{equation}And thus,  \begin{equation}\label{adapted}\phi(z_0)=0,\,\, \d\phi(z_0)=0,\, \d^2\phi(z_0)=0,\,\,\d\dbar\phi(z_0)=1,\,\, \omega(z_0)=d\ell_z .\end{equation}
%In general, $\phi$ contains a%
%pluriharmonic term $f(z) +\overline{f(z)}$, but  a change of frame for $L$
%eliminates that term up to fourth order.
An example on the \kahler normal coordinate and the adapted frame is present in \S \ref{su2}, it is the affine coordinate for the Fubini-Study metric of the hyperplane line bundle over the complex projective space $(\mathcal O(1), h_{FS})\to (\CP,\omega_{FS})$. We also refer to \S 3.1 in \cite{DSZ1} for more details.

After we choose the normal coordinate at $z_0$, by identities \eqref{adapted}, the joint density of the Gaussian processes $(p_n, q_n, r_n)$ (recall \eqref{dddddj}) at $z_0$ should be the same as the joint density of Gaussian processes $(f^p_n, \d f^p_n, \d^2 f^p_n)$.
Hence, by \eqref{adapted} again, the local expression (\ref{d1}) admits the following global expression,
$$\e\left(\sum_{z\in \Omega_p}\delta_{z}\right)=\left(\int_{\C^2} p^n_{z}(y,0,\xi)\left| |\xi|^2-n^2   |y|^2 \right|d\ell_yd\ell_\xi\right)\omega:= K_n(z|p),$$
where  $p^n_{z}(y,s,\xi)$ is the joint density of the Gaussian processes $(f_n^p,\frac{\d f_n^p}{\d z},\frac{\d ^2 f_n^p}{\d z^2})$.
This completes the proof of Lemma \ref{joint1}.
\end{proof}
%\subsection{Covariance matrix and joint density}
% By the relation we get about the conditional critical points,
% we write $$s_n^p=\sum_{j=1}^{d_n-1}a_j f_{j,n}^pe^{-\frac n2\phi}$$
%with $s_{j,n}=f_{j,n}e^{-\frac n2\phi}$ is an orthogonal basis of $\hnp$.
%Denote $$f_n^p=\sum_{j=1}^{d_n-1}a_j f_{j,n}^p $$
%Now we can compute the joint density of
 \subsection{Proof of Theorem \ref{thm1} }
  In this subsection, we will calculate the joint density $p_z^n$ of the Gaussian processes of $(f^p_n, \d f^p_n, \d^2 f^p_n)$.

 For the conditional Gaussian processes $(f^p_n, \d f^p_n, \d^2 f^p_n)$, the joint density is given by the formula \cite{AT}
 \begin{equation} p_z^n(y,s,\xi)=\frac{1}{\pi^3}\frac{1}{\det \Delta_n}\exp \left\{\left \langle \begin{pmatrix}y\\ s\\ \xi
 \end{pmatrix}, (\Delta_n)^{-1} \begin{pmatrix}\bar y\\ \bar s\\ \bar\xi
 \end{pmatrix}\right \rangle
    \right\},
 \end{equation}

  where $\Delta_n$ is the covariance matrix of the conditional Gaussian process $(f^p_n, \d f^p_n, \d^2 f^p_n)$.

We rearrange the order of the Gaussian processes and write $\tilde\Delta_n$ as the covariance
matrix of  $(\d f^p_n, \d^2 f^p_n,f^p_n)$, then we rewrite
 \begin{equation} p_z^n(y,s,\xi)=\frac{1}{\pi^3}\frac{1}{\det \tilde\Delta_n}\exp \left\{\left \langle \begin{pmatrix}s\\ \xi\\ y
 \end{pmatrix}, (\tilde\Delta_n)^{-1} \begin{pmatrix}\bar s\\ \bar \xi\\ \bar y
 \end{pmatrix}\right \rangle
    \right\}.
 \end{equation}

  The covariance matrix is then given by
    \begin{equation}\tilde \Delta_n=\begin{pmatrix}A_n& B_n\\ B_n^*& C_n
    \end{pmatrix}_{3\times 3},\end{equation}

   where $$A_n=\d_{  z}\d _{\bar w}\Pi_n^p(z,w)|_{z=w},$$
    $$B_n=(\d_{  z}\d^2_{\bar w}\Pi_n^p(z,w)|_{z=w} ,\d_{  z}\Pi_n^p(z,w)|_{z=w}),$$
   and
   $$C_n=\begin{pmatrix}\d_{  z}^2\d_{\bar w}^2\Pi_n^p(z,w)|_{z=w}& \d_{  z}^2\Pi_n^p(z,w)|_{z=w}\\ \d^2_{\bar w}\Pi_n^p(z,w)|_{z=w}&\Pi_n^p(z,z) \end{pmatrix},$$

  where %the conditional Bergman kernel \begin{equation}\label{cova}\Pi_n^p(z,w):=\sum_{j=1}^{d_n-1}f^p_{n,j}(z)\overline{f^p_{n,j}(w)}\end{equation}
 $\Pi_n^p(z,w)$ is the covariance kernel of the conditional Gaussian process $f^p_n$.

   Thus, when $s=0$, by some element matrix computations, we have,
\begin{lem} \label{pz}
With all notations above,
\begin{equation}\label{denty}p^n_z(y,0,\xi) =\frac 1{\pi^{3} }\frac 1{ A_n \det \Lambda_n}\exp\left\{-\left\langle \begin{pmatrix} \xi  \\y\end{pmatrix}, \Lambda_n^{-1}\begin{pmatrix} \bar \xi \\ \bar y \end{pmatrix}\right \rangle\right\},
\end{equation}
where
\begin{equation}\Lambda_n=C_n- A_n^{-1}B_n^* B_n.\end{equation}
   \end{lem}

We combine  Lemma \ref{joint1} and Lemma \ref{pz} to rewrite,
$$K_n(z|p)=   \left(\int_{\C^2} \frac 1{\pi^{3} }\frac 1{ A_n \det \Lambda_n}\exp\left\{-\left\langle \begin{pmatrix} \xi  \\y\end{pmatrix}, \Lambda_n^{-1}\begin{pmatrix} \bar \xi \\ \bar y \end{pmatrix}\right \rangle\right\}\left| |\xi|^2-n^2   |y|^2 \right|d\ell_yd\ell_\xi\right)\omega.$$
  Hence, we will complete the proof of Theorem \ref{thm1} once we  derive the expression of the  covariance kernel $\Pi_n^p(z,w)$ of the conditional Gaussian process. This is computed in the next subsection.
\subsection{Covariance kernel}\label{djdjdjd}
%In this subsection, we recall some basic properties of Bergman kernel $\Pi_n(z,w)$ and derive
%the formula for the covariance kernel $\Pi_n^p(z,w)$ \eqref{cova}. We refer to \cite{SZZ,Z} %for more details.

The Bergman kernel is the orthogonal projection from the $L^2$ integral
sections to the holomorphic sections,
\begin{equation}\label{projectionof}\Pi_n(z,w): L^2(M, L^n)\to \hn
\end{equation}
with respect to the inner product \eqref{innera}.
It has the following reproducing property
\begin{equation}\label{repo} \langle s_n(z),\Pi_n(z,w) \rangle_{h^n}=s_n(w),
\end{equation}
where $s_n \in \hn$ is a global holomorphic section.
Let $\{s_{n,1}, ..., s_{n,d_n}\}$
 be any orthonormal basis of  $\hn$ with respect to the inner product \eqref{innera}, then we have, %write locally $s_{n,j}=f_{n,j}e^{\otimes n}$, then
\begin{equation}\label{bergman1} \Pi_n(z,w)=\sum_{j=1}^{d_n} s_{n,j}(z)\otimes \overline{s_{n,j}(w)}.
%=\left(\sum_{j=1}^{d_n} f_{n,j}(z)\overline{f_{n,j}(w)}\right)e^{\otimes n}(z)\otimes \overline {e^{\otimes n}(w)}
\end{equation}
It's easy to check that $\Pi_n(z,w)$ is also the covariance kernel of the Gaussian process $(\hn,d\gamma_{d_n})$ defined by \eqref{standardgauss}.
%Let  $\{s_1 , ..., s_{d_n} \}$ be an orthonormal basis of $\hn$ with respect to the inner product
%\eqref{innera}, then we have global expression  \begin{equation}
%\Pi_n(z,w)=\sum_{j=1}^{d_n}s_j(z)\otimes \overline{s_j(w)}
%\end{equation} out
%For simplicity, throughout the article, we will discard the frame $e^{\otimes n}(z)\otimes \overline {e^{\otimes n}(w)}$ and write the Bergman kernel locally as %$\Pi_n(z,w)=\sum_{j=1}^{d_n} f_{n,j}(z)\overline{f_{n,j}(w)}$.

 Recall $H^p_0\subset H^0(M, L^n)$ is the space of holomorphic sections vanishing at $p$.
 Let $\{s_{n,1}^p, ..., s_{n,d_n-1}^p\}$
 be an orthonormal basis of  $H^p_0$ with respect to the inner product \eqref{innera}.
 By the reproducing property of the Bergman kernel $\Pi_n(z,w)$, one can show that the holomorphic sections $\{s_{n,1}^p, ..., s_{n,d_n-1}^p, \Phi_n^p\}$
 will be an orthonormal basis for $H^0(M, L^n)$ (see equation (3.7) in \cite{SZZh} for more details) where (by discarding the local frame $e^{\otimes n}$)\begin{equation}
 \Phi_n^p(z)=\frac{\Pi_n(z,p)}{\sqrt{\Pi_n(p,p) }}.
 \end{equation}

%This can be tell ,
%$$\langle \Phi_n(z,p),s_{n,j}\rangle=\frac{s_{n,j}(p)}{\sqrt{\Pi_n(p,p)}}=0$$
%for all $j=1,\cdots, d_n-1$.
%Furthermore, the $L^2$ norm of $\Phi_n$ is $1$.

Recall relation \eqref{relation}, then the covariance kernel of the conditional Gaussian process is

%Then %by discarding the local frame $e^{\otimes n}(z)\otimes \overline {e^{\otimes n}(w)}$ again,
%the conditional Bergman projection $\Pi_n^p$ \eqref{cova} is,
 \begin{equation}\label{ddd}\Pi_n^p(z,w)=\Pi_n(z,w)-\Phi_n^p(z)\overline{\Phi_n^p(w)}. \end{equation}
  % One can get the conditional density of critical points by applying the conditional Bergman kernel %to the generalized Kac-Rice formula. The difficult part is the rescaling limit of $K(p+\frac u{\sqrt %n}| p)$. In order to get the rescaling limit, one needs to study the asymptotic expansions of the %derivatives of $\Pi_n^p(\frac z{\sqrt n}, \frac w {\sqrt n})$ up to order $4$. One can follow the %estimates in \cite{DSZ2}.
  Hence, we complete the proof of Theorem \ref{thm1}.
\subsection{Further simplification}
 We can further simplify the expression of $K_n(z|p)$ in Theorem \ref{thm1} as follows. Let $H=(\xi, y)$, then we can rewrite
$$ \int_{\C^2} \frac 1{\pi^{3} }\frac 1{ A_n \det \Lambda_n}\exp\left\{-\left\langle \begin{pmatrix} \xi  \\y\end{pmatrix}, \Lambda_n^{-1}\begin{pmatrix} \bar \xi \\ \bar y \end{pmatrix}\right \rangle\right\}\left| |\xi|^2-n^2   |y|^2 \right|d\ell_yd\ell_\xi$$as

$$ \frac 1{\pi^{3} }\frac 1{ A_n \det \Lambda_n}\int_{\C^2}e ^{-H\Lambda_n^{-1} H^*}|HQ_nH^*| d\ell_H, $$

    where $Q_n=\begin{pmatrix}1 & 0\\ 0& -n^2\end{pmatrix}$ and $d\ell_H$ is the Lebesgue measure on $\C^2$.

    We change variable $H\to H\Lambda_n^{-\frac12}$ to get

$$ \frac 1{\pi^{3} }\frac 1{ A_n  }\int_{\C^2}e ^{-  HH^*}|H\Lambda_n^{\frac12} Q_n\Lambda_n^{\frac 12} H^*| d\ell_H. $$

    We diagonalize  $\Lambda_n^{\frac12} Q_n\Lambda_n^{\frac 12}$ with eigenvalues $\lambda_1^n$, $\lambda_2^n$ (note that it is easy to check $\lambda_1^n$ and $\lambda_2^n$ are also eigenvalues of $\Lambda_n Q_n$) to simplify the above integration as

\begin{align*}&\frac 1{\pi^{3} }\frac 1{ A_n  }\int_{\C^2}|\lambda^n_1|y|^2+\lambda_2^n|\xi|^2 | e ^{-   |y|^2-|\xi|^2}d\ell_yd\ell_\xi \\=&\frac 1{\pi  }\frac 1{ A_n  }\int_{0}^\infty\int_0^\infty|\lambda_1^nx+\lambda_2^ny | e ^{-x-y}dxdy \\=&\frac 1{\pi  }\frac 1{ A_n  }\frac{(\lambda_1^n)^2+(\lambda_2^n)^2}{|\lambda_1^n|+|\lambda_2^n|}.\end{align*}

Thus we have,
\begin{lem}\label{eign}The conditional expectation is  \begin{equation}\label{formula}K_n(z|p)=\frac 1{\pi  }\frac 1{ A_n  }\frac{(\lambda_1^n)^2+(\lambda_2^n)^2}{|\lambda_1^n|+|\lambda^n_2|} \omega(z),\end{equation} where $\lambda_1^n$ and $\lambda_2^n$ are eigenvalues of $(\Lambda_nQ_n)(z)$. \end{lem}

%The rescaling limit of the volume form is
%$dV_\omega(p+\frac u{\sqrt n})=\frac 1n \frac {1}{2i}du\wedge d\bar u$, we only need to study the %rescaling limit of $\frac 1{\pi  }\frac 1{ A_n  }\frac{\lambda_1^2+\lambda_2^2}{|\lambda_1|+|\lambda|%_2}(p+\frac u{\sqrt n})$.

\section{Calculations of Theorem \ref{main} for Gaussian random $SU(2)$ polynomials}\label{su2}
In this section, we will derive the rescaling conditional density of critical
points for Gaussian random $SU(2)$ polynomials. This is the case where $M = \CP\cong S^2$ and $L$ is the hyperplane
line bundle $\mathcal O(1)$. The global holomorphic sections of  $\mathcal O(1)$ are linear functions on $\C^2$ and hence the global holomorphic sections of
$L^
n = \mathcal O(n)$ are homogeneous polynomials of degree $n$.

The \kahler form on $\CP$ is the Fubini-Study form.  In an affine coordinate, the \kahler form and the \kahler potential for the Fubini-Study metric are
\begin{equation}\label{fsadapt}
 \omega_{FS}=\frac{\sqrt{-1}}{2}\frac{dz\wedge d \bar z}{(1+|z|^2)^2},\,\,\,\phi(z)=\log (1+|z|^2).
\end{equation} It's easy to check that $\phi$ satisfies \eqref{adapted} and the affine coordinate is actually the \kahler normal coordinate at $z_0=0$.

We equip
$\mathcal O(1)$ with its
Fubini-Study metric. In fact,  we can choose an adapted frame $e(z)$ such that
 $$|e(z)|^2_{h_{FS}}=e^{-\phi}=\frac 1{1+|z|^2}.$$

%$h_{FS}$ given by
%$$
%\|s\|_{h_{FS}}([w])=\frac{|(s,w)|}{|w|},\,\,\,\, w=(w_0,w_1)\in \C^2
%$$
%for $s\in (\C^2)^*\cong H^0(\CP, \mathcal O(1))$ where $|w|^2=|w_0|^2+|w_1|^2$ and
%$[w]\in\CP$ is the complex line through $w$.
%$$
%\omega_{FS}=\frac{\sqrt{-1}}{2}\d\dbar \log |w|^2.
%$$

Hence,  an orthonormal basis of $H^0(\CP,\mathcal O(n))$ under the inner product \eqref{innera} is given by
$$\label{basis1}
\left\{\left(\sqrt{(n+1) {n\choose j}}z^j\right) e^{\otimes n}\right\}_{j=0}^{n}.
$$
Throughout the article, we will discard the local frame $e^{\otimes}$ for simplicity.

The Gaussian linear combination of the above basis is Gaussian random $SU(2)$ polynomials and the distribution of zeros of such polynomials is invariant under the rotation of $S^2$ \cite{H}.

By \eqref{bergman1}, the Bergman kernel for the Fubini-Study case is
$$
\Pi_n^{SU(2)}(z,w)=(n+1) (1+z\bar w)^n.
$$
By the expression of $K_n(z|p)$ in Theorem \ref{thm1}, the expected density of critical points is unchanged when the covariance kernel  is multiplied by a constant (or equivalently the Gaussian process is multiplied by a constant).
In the following computations, for simplicity, we can replace the Bergman kernel $\Pi^{SU(2)}_n(z,w)$ by the normalized Bergman kernel
   $$\Pi_n(z,w)=(1+z\bar w)^n.$$

     By formula \eqref{ddd},
   % \begin{equation}\Pi_n^p(z,w)=\Pi_n(z,w) -\frac{\Pi_n(z,p)\overline{\Pi_n(w,p) }}%{\Pi_n(p,p)}=\Pi_n(z,w) -\frac{\Pi_n(z,p) \Pi_n(p,w)  }{\Pi_n(p,p)},\end{equation}
 we have the following expression for the covariance kernel of the conditional Gaussian measure $$\Pi_n^p(z,w)=(1+z\bar w)^n -\frac{(1+z\bar p)^n (1+p\bar w)^n  }{(1+p\bar p)^n}.$$

Now let's compute the matrices $B_n$ and $C_n$ for $H^0(\CP, \mathcal O(n))$. Indeed, we have,
   $$\frac{\d \Pi_n^p}{\d z}=n(1+z\bar w)^{n-1}\bar w-\frac{n \bar p (1+z\bar p)^{n-1}(1+p\bar w)^{n }}{(1+p\bar p)^n}$$

   $$\frac{\d^2 \Pi_n^p}{\d z\d \bar w}=n(1+z\bar w)^{n-1}+zn(n-1)(1+z\bar w)^{n-2} \bar w-\frac{n^2 p\bar p (1+z\bar p)^{n-1}(1+p\bar w)^{n-1 }}{(1+p\bar p)^n}$$

   $$\frac{\d^2 \Pi_n^p}{\d^2 z}=n(n-1)(1+z\bar w)^{n-2}\bar w^2-\frac{n(n-1) \bar p^2 (1+z\bar p)^{n-2}(1+p\bar w)^{n }}{(1+p\bar p)^n}$$

    $$\frac{\d^3 \Pi_n^p}{\d^2 z\d \bar w}=2n(n-1)(1+z\bar w)^{n-2}\bar w +n(n-1)(n-2)(1+z\bar w)^{n-3}z\bar w^2-\frac{n^2(n-1) p\bar p^2 (1+z\bar p)^{n-2}(1+p\bar w)^{n-1}}{(1+p\bar p)^n}$$

    $$\frac{\d^4 \Pi_n^p}{\d^2 z\d^2 \bar w}=2n(n-1)(1+z\bar w)^{n-2}  +4n(n-1)(n-2)(1+z\bar w)^{n-3}z\bar w+n(n-1)(n-2)(n-3)(1+z\bar w)^{n-4}z^2\bar w^2$$$$ -\frac{n^2(n-1)^2 p^2\bar p^2 (1+z\bar p)^{n-2}(1+p\bar w)^{n-2}}{(1+p\bar p)^n}$$

     % choosing local coordinate and adapted frame at fixed $z$ with $p=z+\frac u{\sqrt n}$.
     Throughout the article, the notation $a_n\sim b_n$
     means the asymptotics $a_n=b_n+o(b_n)$ as $n$ large enough; for simplicity we will discard the negligible terms $o(b_n)$ in some steps
     which do not contribute in the pointwise limit as $n\to\infty$ in order to keep track of the leading order term, although the precise estimates for all errors terms can be derived.

     In order to get the rescaling density $K_n(p+\frac{u}{\sqrt n}|p)$ around $p$, %it's equivalent to study $K(z|z+\frac u{\sqrt{n}})$ by changing variable.
      we choose the affine coordinate at $p=0$ with $z=\frac{u}{\sqrt n}$. By Lemma \ref{eign}, we need to
     find rescaling limits of $\lambda_1(\frac{u}{\sqrt n})$ and $\lambda_2(\frac{u}{\sqrt n})$ where $\lambda_1$ and $\lambda_2$ are eigenvalues of matrix $\Lambda_nQ_n$, or equivalently, we need to find the estimates of two eigenvalues of matrix $(\Lambda_nQ_n)(\frac{u}{\sqrt n})$.

     We first have the following asymptotics as $n$ large enough,

    \begin{equation}\label{aaaaa}A_n(\frac{u}{\sqrt n})=n(1+\frac{|u|^2}n)^{n-1}+(n-1)|u|^2(1+\frac{|u|^2}n)^{n-2}\sim n(1+|u|^2)e^{|u|^2}.\end{equation}
Similarly, we have,

    $$B_n(\frac{u}{\sqrt n})\sim (2n^{\frac32}u+n^{\frac32}u|u|^2, n^{\frac12}\bar u)e^{|u|^2},$$
and
     $$C_n(\frac{u}{\sqrt n})\sim e^{|u|^2}\begin{pmatrix}  2n^2+4n^2|u|^2+n^2|u|^4&    n\bar u^2 \\   n  u^2 &1-e^{-|u|^2} \end{pmatrix}.$$

     Since $\Lambda_n= C_n-A_n^{-1}B_n^*  B_n $, it's easy to compute

    $$(\Lambda_nQ_n)(\frac{u}{\sqrt n})\sim \frac{\ n^2e^{|u|^2}}{1+|u|^2}\times \begin{pmatrix}2+2|u|^2+|u|^4&0\\
     0& -1+e^{-|u|^2}+|u|^2e^{-|u|^2}\end{pmatrix}. $$

      %We apply the asymptotic expansion $(1+\frac{|u|^2}{n})^{-n}\sim e^{-|u|^2}$ as $n\to\infty$,
        Hence, the eigenvalues of $(\Lambda_nQ_n)(\frac{u}{\sqrt n})$ satisfy asymptotics,    \begin{equation}\label{aaat}\lambda_1(\frac{u}{\sqrt n})\sim \frac{\ n^2(2+2|u|^2+|u|^4)e^{|u|^2}}{1+|u|^2}>0\end{equation} and    \begin{equation}\label{att}\lambda_2(\frac{u}{\sqrt n})\sim \frac{\ n^2(-1+e^{-|u|^2}+|u|^2e^{-|u|^2})e^{|u|^2}}{1+|u|^2}\leq 0.\end{equation}

For the Fubini-Study metric, we have the following estimate,  \begin{equation}\label{fubini}
\lim_{n\to\infty}n\omega_{FS}(\frac{u}{\sqrt n})=\lim_{n\to\infty}n\frac{\sqrt{-1}}2\frac{d\frac{u}{\sqrt n}\wedge d\frac{\bar u}{\sqrt n}}{(1+|\frac{u}{\sqrt n}|^2)^2}=\frac{\sqrt{-1}}{2}du\wedge d\bar u.\end{equation}

As a remark, this estimate is true for any \kahler metric $\omega$ by \eqref{adapted}, i.e.,
\begin{equation}\lim_{n\to\infty}n\omega(p+\frac{u}{\sqrt n})=\frac{\sqrt{-1}}{2}du\wedge d\bar u.\end{equation}
If we combine Lemma \ref{eign} with asymptotics \eqref{aaaaa}\eqref{aaat}\eqref{att}\eqref{fubini}, we have the limit,
$$
\lim_{n\to\infty} K_n(p+\frac{u}{\sqrt n}|p)=\frac{1}{\pi a_\infty^2}\frac{(\lambda_1^\infty)^2+(\lambda_2^\infty)^2}{|\lambda_1^\infty|+|\lambda_2^\infty|}\frac{\sqrt{-1}}{2}du\wedge d\bar u
$$
with
$a_\infty$, $\lambda_1^\infty$
and $\lambda_2^\infty$ given in Theorem \ref{main}.
 Hence we prove Theorem \ref{main} for Gaussian random $SU(2)$ polynomials.

\begin{rem}  In \cite{SZZh}, the authors studied the rescaling limit of the expected density of zeros given that the random sections vanish at a point. The expected (conditional) density of zeros of Gaussian random holomorphic functions can be derived by the probabilistic Poincar\'{e}-Lelong formula (see \S \ref{profooo}). In fact, for Gaussian random  $SU(2)$ polynomials, we have the following explicit global expression,
\begin{align*}&\e_{d\gamma_{d_n}}^{SU(2)} (\sum_{z:\,s_n(z)=0}\delta_z|s_n(p)=0)\\&=\frac{\sqrt{-1}}{2\pi}\d\dbar\log \Pi_n^p(z,w)\\ &=\frac{\sqrt{-1}}{2\pi}\d\dbar\log \left((1+z\bar w)^n -\frac{(1+z\bar p)^n (1+p\bar w)^n  }{(1+p\bar p)^n}\right).\end{align*}
Hence, the rescaling limit of the above density by choosing the affine coordinate at $p=0$ with $z=\frac u{\sqrt n}$ will be
$$\frac{\sqrt{-1}}{2\pi}\d\dbar \log (e^{|u|^2}-1),$$
which is one of the main results Corollary 1.3 of \cite{SZZh}.

\end{rem}

  % Now let's compute the meromorphic connection.

% We have,  the density is given by $$\d\dbar \log \frac{\d^2 \Pi_n^p(z,w)}{\d z\d\bar w}|_{z=w}=\frac{2-2|u|^2e^{-|u|^2}-|u|^4e^{-|u|^2}}{(1-|u|^2e^{-|u|^2})^2}$$

 %we have more attraction compared with smooth case.

\section{Proof of Theorem \ref{main}}\label{profooo}
In this section, we will prove Theorem \ref{main} for any Riemann surfaces.
%we will first recall some asymptotic properties of Bergman kernel and its rescaling %properties, the estimates of those asymptotics are key steps in order to.

By Lemma \ref{pz}, the joint density $p^n_z$ only depends on
the Bergman kernel and its derivatives up to order $4$; thus the rescaling limit
of the conditional expectation should only depend on the rescaling limits
of Bergman kernel and its derivatives.
We will see that all these rescaling limits are universal.
% and can be expressed as the
%Szeg\"o kernel $\mathbf H _1^H(z, w)$ of level one for the reduced Heisenberg group $\mathbf H %^n_{
%red}$.

%\subsection{Bergman kernel: on diagonal expansion and scaling limits}
%Recall that the Bergman kernel is the orthogonal projection from the $L^2$ integral
%sections to the holomorphic sections
%\begin{equation}\Pi_n(z,w): L^2(M, L^n)\to \hn
%\end{equation}
%with respect to the inner product \eqref{innera}.
The Bergman kernel has the following  Tian-Yau-Zelditch $C^\infty$-expansion on the diagonal for
Riemann surfaces \cite{MM1, T,Y,Ze2},
\begin{equation}\label{full}
\Pi_n(z,z)=ne^{n\phi}(1+a_1(z)n^{-1}+a_2(z)n^{-2}+\cdots),
\end{equation}
where $a_1$ is the scalar curvature of $\omega$.

Integrating over $M$ with respect to $e^{-n\phi}\omega$
 gives
the well-known dimension polynomial,
 \begin{equation}
 d_n=n(1+n^{-1}\int_M a_1 \omega+n^{-2}\int_M a_2\omega+\cdots ).
 \end{equation}

The proof of the full expansion \eqref{full} makes use of Boutet de Monvel-Sjostrand parametrix
construction \cite{Ze2}. Actually the same construction can be carried out to derive the recaling limits of the Bergman kernel off diagonal. We also remark that the estimates of the Bergman kernel off diagonal are studied by Dai-Liu-Ma \cite{DLM}.

Let's choose the \kahler normal coordinate at $p$. First, if we apply identities \eqref{adapted}, we have the following on diagonal asymptotics at $p$,
\begin{equation}\label{p}
\Pi_n(p,p)=n(1+a_1(p)n^{-1}+a_2(p)n^{-2}+\cdots).
\end{equation}

Regarding the rescaling limit of the Bergman kernel, we have the following universal limit,
 \begin{equation}\label{offdiag}
 \Pi_n(p+\frac{u}{\sqrt n}, p+\frac{v}{\sqrt n})=n(e^{u \cdot {\bar v}}+\frac{1}{\sqrt n}p_1+...),
 \end{equation}
 where  $p_1$ is a homogeneous polynomial and the error terms have precise estimates (see \S 5 in \cite{SZ1}).
%As a remark, the full expansions  \eqref{full}\eqref{offdiag} are
%differentiable for any times.
\begin{rem}\label{bgj}
The term $ne^{u\cdot \bar v}$ is actually the rescaling limit of the Bergman kernel for the Bargmann-Fock space. We refer to \cite{BSZ1} for more details.
\end{rem}
%By taking derivatives on both sides of \eqref{full}, we have full expansion,
%$$\d_z \Pi_n(z,z)=n^2\d_z\phi e^{n\phi}(1+n^{-1}a_1+\cdots)+e^{n\phi}(\d_za_1+n^{-1}\d_za_2+\cdots).$$

%We choose adapted frame at $z=0$, by identities \eqref{adapted},  we have the estimate,
%\begin{equation}\d_z \Pi_n(0,0)=\d_za_1+O(n^{-1}).\end{equation}
%In the following computations, $O(n^{k})$
%means a function whose $C^m$-norm is $O(n^{k})$
%in the standard
%sense for all $m$.

%Similarly, we can take more derivatives on both sides of \eqref{full}, after we apply the identities \eqref{adapted},  we have estimates

%\begin{equation}\Pi_n(0,0)=n+O(1), \end{equation}
%\begin{equation}\d_z^2\Pi_n(0,0)=\d^2 a_1+O(n^{-1}),\end{equation}
%\begin{equation}\label{ddbar}\d_z\d_{\bar z} \Pi_n(0,0)=n^2+na_1+(\d\dbar a_1+a_2)+O(n^{-1}),\end{equation}
%\begin{equation}\d^2_z\d_{\bar z} \Pi_n(0,0)=2n\d a_1+(2\d a_2+\d^2\dbar a_1)+O(n^{-1}), \end{equation}

%and the full expansion \begin{equation}\d^2_z\d_{\bar z} ^2\Pi_n(0,0)=2n^3(1+n^{-1}a_1+\cdots)+4n\d\dbar a_1+n^2R_{11\bar 1\bar 1}(1+n^{-1}a_1+\cdots),\end{equation}
%where $R_{11\bar 1\bar 1}$ is the Ricci curvature.

Regarding the rescaling limit of the Bergman kernel and its derivatives on the diagonal, we have the following estimates (we refer \cite{B, BSZ1} for more details),

  \begin{align*}\Pi_n(\frac{u}{\sqrt n},\frac{u}{\sqrt n})&=ne^{|u|^2}+O(n^{\frac12}),\\
(\d\Pi_n)(\frac{u}{\sqrt n},\frac{u}{\sqrt n})&=n^{\frac 32}\bar ue^{|u|^2}+O(n),\\
(\d^2\Pi_n)(\frac{u}{\sqrt n},\frac{u}{\sqrt n})&=n^{2}\bar u^2e^{|u|^2}+O(n^{\frac32}),\\
(\d\dbar\Pi_n)(\frac{u}{\sqrt n},\frac{u}{\sqrt n})&=n^2e^{|u|^2}(1+|u|^2)+O(n^{\frac 32}),\\
(\d^2\dbar \Pi_n)(\frac{u}{\sqrt n},\frac{u}{\sqrt n})&=n^{\frac 52}\bar ue^{|u|^2}(2+|u|^2)+O(n^2), \\
(\d^2\dbar^2 \Pi_n)(\frac{u}{\sqrt n},\frac{u}{\sqrt n})&=n^{3}e^{|u|^2}(2+4|u|^2+|u|^4)+O(n^{\frac52}).   \end{align*}
And similarly, we have,
  \begin{align*}\Pi_n(\frac{u}{\sqrt n},0)&=n+O(n^{\frac12}),\,\,\,
(\d\Pi_n)(\frac{u}{\sqrt n},0)=O(n),\,\,\,
(\d^2\Pi_n)(\frac{u}{\sqrt n},0)=O(n^{\frac 32}).   \end{align*}

%\begin{equation}\label{eq1}(\d\Pi_n)(\frac{u}{\sqrt n},0)=\d_{\frac{u}{\sqrt n}}\Pi_n(0,\frac{u}{\sqrt n})=\sqrt n\d_u\Pi_n(0,\frac{u}{\sqrt n})=n^{\frac 32 }\bar u+O(n^{}),\end{equation} and similarly, \begin{equation}\label{eq2}(\d^2 \Pi_n)(0,\frac{u}{\sqrt n})=n^{2 }{\bar u}^2+O(n^{ \frac 32 }).\end{equation}

Now we can get the estimates of the covariance matrix, %We choose coordinate at $z=0$ with $p=\frac{u}{\sqrt n}$,
  \begin{align*}A_n(\frac{u}{
  \sqrt n})=&\d_{ z}\d _{\bar w}\Pi_n^p(z,w)|_{z=w=\frac{u}{\sqrt n},0}\\
 =& \d_{  z}\d _{\bar w}\Pi_n (z,w)|_{z=w=\frac{u}{\sqrt n}}-\frac{\d_z \Pi(z,p)\overline{\d_w \Pi(w,p)}}{\Pi_n (p,p)}|_{z=w=\frac{u}{\sqrt n},p=0}\\=&n^2e^{|u|^2}(1+|u|^2)+O(n^{\frac 32})-\frac{O(n^2)}{n(1+a_1(p)n^{-1}+O(n^{-2}))}\\=& n^2(1+|u|^2)e^{|u|^2}+O(n^{\frac32})
  \end{align*}

The similar computations yield,
$$B_n(\frac{u}{\sqrt u})= \left(n^{\frac52} u(2+|u|^2)e^{|u|^2}+O(n^{2}), n^{\frac{3}{2}}   \bar ue^{|u|^2}+O(n^{ })\right) $$
and
 $$C_n(\frac{u}{\sqrt u})=\begin{pmatrix}  n^3(2+4|u|^2+|u|^4)e^{|u|^2}+O(n^{\frac52})&    n^2\bar u^2 e^{|u|^2}+O(n^{\frac32})\\ n^2 u^2 e^{|u|^2}+O(n^{\frac32})&n(e^{|u|^2}-1)+O(n^{\frac12}) \end{pmatrix}.$$
Note that the above estimates are the same as the ones in \S \ref{su2} (except an extra factor $n$ since we used normalized Bergman kernel in \S \ref{su2}),
hence Theorem \ref{main} follows the same computations as in  \S \ref{su2} by finding the estimates of the eigenvalues of $( \Lambda_nQ_n)(\frac u{\sqrt n})$.
%Thus the estimates of the eigenvalues of $( \Lambda_nQ_n)(\frac u{\sqrt n})$ should be the same as the ones in Gaussian random $SU(2)$ polynomials, which completes the proof of.

  %   Since $\Lambda_n= C_n-B_n^* A_n^{-1} B_n $,  by discarding negligible terms, we have asymptotics

   % $$(\Lambda_nQ_n)(\frac{u}{\sqrt n})\sim \frac1{1-|u|^2e^{-|u|^2}}\times$$$$ \begin{pmatrix} 2n^3(1-|u|^2 e^{-|u|^2})-n^3|u|^4e^{-|u|^2}& -n^2\bar u^2e^{-|u|^2}\\
    %  -n^2 u^2e^{-|u|^2}& -n^3(1-|u|^2 e^{-|u|^2}-e^{-|u|^2})\end{pmatrix}.$$

%     The eigenvalues of $(\Lambda_nQ_n)(\frac{u}{\sqrt n})$ satisfy asymptotics, \begin{equation}\lambda_1(\frac{u}{\sqrt n})\sim \frac{n^3(2-2|u|^2e^{-|u|^2}-|u|^4e^{-|u|^2})}{1-|u|^2e^{-|u|^2}}\end{equation} and \begin{equation}\lambda_2(\frac{u}{\sqrt n})\sim \frac{-n^3(1-|u|^2e^{-|u|^2}-e^{-|u|^2})}{1-|u|^2e^{-|u|^2}}.\end{equation}

%For the \kahler form under the adapted frame at $z=0$, we have \begin{equation}
%\lim_{n\to\infty}n\omega(\frac{u}{\sqrt n})=\frac{\sqrt{-1}}{2}du\wedge d\bar u.
%\end{equation}
%Combining these asymptotics with Lemma \ref{eign}, we have the limit
%\begin{equation}
%\lim_{n\to\infty} K_n(p+\frac{u}{\sqrt n}|p)=\frac{1}{\pi a_\infty^2}\frac{(\lambda_1^\infty)^2+(\lambda_2^\infty)^2}{\lambda_1^\infty+\lambda_2^\infty}\frac{\sqrt{-1}}{2}du\wedge d\bar u,
%\end{equation}
%where
%$$a_\infty=1-|u|^2e^{-|u|^2},$$ $$\lambda_1^\infty=  2-2|u|^2e^{-|u|^2}-|u|^4e^{-|u|^2}  $$
%and $$\lambda_2^\infty=  1-|u|^2e^{-|u|^2}-e^{-|u|^2}.$$
%Hence

\section{Proofs of Theorem \ref{thm2}}\label{proofoooo}
In this section, we will apply the probabilistic Poincar\'{e}-Lelong formula to derive a global formula for $D_n(z|q)$ of the empirical measure of zeros with a conditioning critical point.
% then we rescale this formula to get the local behavior between zeros and the conditional critical point.

\subsection{Poincar\'{e}-Lelong formula}
Given a global holomorphic sections $s_n$ of  a positive holomorphic line bundle over \kahler manifolds, we denote $Z_{s_n}$ as the empirical measure of zeros of $s_n$. We write locally $s_n=f_ne^{\otimes n}$, then the classical Poincar\'{e}-Lelong formula states that \cite{GH}
\begin{equation}Z_{s_n}=\frac{\sqrt{-1}}{2 \pi}\partial\bar\partial \log |f_n|^2.\end{equation}
%$$Z_{s_n}=\frac{\sqrt{-1}}{2 \pi}\partial\bar\partial \log |s_n|_{h^n}^2+\frac{n\omega}{ \pi}.$$

Taking the expectation on both sides, we have the following  probabilistic Poincar\'{e}-Lelong formula \cite{HKPV, SZZh}: Let $S\subset \hn$ be a Gaussian random field with covariance kernel $\Pi_S(z,w)$, then
\begin{equation}\label{pb}\e_S(Z_{s_n})=\frac{\sqrt{-1}}{2 \pi}\partial\bar\partial \log |\Pi_S(z,z)|.\end{equation}

%$$\e_S(Z_{s_n})=\frac{\sqrt{-1}}{2 \pi}\partial\bar\partial \log |\Pi_S(z,z)|_{h^n}^2+\frac{n\omega}{ \pi}.$$

\subsection{Proof of Theorem \ref{thm2}}\label{proofof2}
Now we turn to prove Theorem \ref{thm2}. Recall \eqref{alternative1}, we rewrite the conditional expectation of zeros $D_n(z|q)$ as

\begin{equation}D_n(z|q)=  \e_{\left(\hnq, d\gamma^q_{d_n-1}\right)}\left( Z_{s_n}\right). \end{equation}

  By probabilistic Poincar\'{e}-Lelong formula \eqref{pb}, we have,

\begin{equation}D_n(z|q)=  \e_{\left(\hnq, d\gamma^q_{d_n-1}\right)}\left( Z_{s_n}\right)=\frac{\sqrt{-1}}{2 \pi}\partial\bar\partial \log |\Pi_n^q(z,z)| . \end{equation}
where $\Pi_n^q(z,w)$ is the covariance kernel of the conditional Gaussian random sections $(\hnq, d\gamma^q_{d_n-1})$,
%which is induced by $(\hn, d\gamma_{d_n})$,
where $\hnq$ is the kernel of the linear map $s_n\to \nabla_{h^n} s_n$ \eqref{linearmap2}.

%To prove Theorem \ref{thm2}, it is enough to find the expression of $\Pi_n^q(z,z)$.
 %As we discuss in \S \ref{ce}, $\hnq$ is the kernel of the smooth Chern connection $\nabla_{h^n}$.
  %Recall that given a holomorphic section $s_n=f_ne^{\otimes n}$, we have the following formula in the local coordinate \cite{GH}
 %\begin{equation}\label{dddddk}\nabla_{h^n}s_n=(\frac{\partial f_n}{\partial z}-nf_n\frac{\partial \phi}{\partial z}) e^{\otimes n}\otimes dz, \end{equation}
 % where $\phi$ is the local \kahler potential.
 The Kodaira embedding implies that $\hnq$ is a subspace of $\hn$ of codimension $1$.
 Let $\{s_{n,1}^q, ..., s_{n,d_n-1}^q\}$
 be an orthonormal basis of  $\hnq$ with respect to the inner product \eqref{innera}. Such basis satisfies $\nabla_{h^n} s_{n,j}(q)=0$ for all $j=1,\cdots,d_n-1$.
  We can extend this basis to be a basis of $\hn$, we denote such basis as $\{s_{n,1}^p, ..., s_{n,d_n-1}^p, \Psi_n^q\}$. By relation \eqref{relation} again, the covariance kernel for the conditional Gaussian measure  $(\hnq, d\gamma^q_{d_n-1})$ is
 \begin{equation}
 \Pi_n^q(z,w)=\Pi_n(z,w)-\Psi_n^q(z)\overline{\Psi_n^q(w)}.
 \end{equation}
 And hence,
\begin{equation}\label{ddddd}D_n(z|q)=\frac{\sqrt{-1}}{2 \pi}\partial\bar\partial \log |\Pi_n(z,z)-|\Psi_n^q(z)|^2| . \end{equation}

   To prove Theorem \ref{thm2}, it is enough to find the expression of $|\Psi_n^q|^2$. In the following computations, we will discard the local frames $e^{\otimes n}$ and $dz$ for simplicity. We write $s^q_{n,j}=f_{n,j}^q e^{\otimes n}$ locally. Then the Bergman kernel reads

  $$\Pi_n(z,w)=\sum_{j=1}^{d_n-1} f^q_{n,j}(z)\overline{f^q_{n,j}(w)}+\Psi_n^q(z)\overline{\Psi_n^q(w)}.$$
We take the Chern connection $\nabla^z_{h^n}$ on both sides with respect to variable $z$ and evaluate at $z=q$,
we have the relation,
  $$\nabla_{h^n}^z\Pi_n(z,w)|_{z=q}=\nabla_{h^n}^z\Psi_n^q(z)|_{z=q}\overline{\Psi_n^q(w)}.$$
This implies that $\Psi_n^q(w)$ is parallel to $\overline{\nabla_{h^n}^z\Pi_n(z,w)|_{z=q}}$. We define $$\Psi_n^q(w)=\lambda_q \overline{\nabla_{h^n}^z\Pi_n(z,w)|_{z=q}}.$$
We will find $|\lambda_q|^2$  in order to get  $|\Psi_n^q(w)|^2$ in \eqref{ddddd}.
By definition of the Chern connection \eqref{dddddk}, we have
 $$\Psi_n^q(w)=\lambda_q \overline{[\frac{\partial\Pi_n(z,w)}{\partial z}|_{z=q}-n\frac{\partial\phi}{\partial z}(q)\Pi_n(q,w)]}.$$
We can choose the
\kahler normal coordinate freezing at the point $z_0=q$ as the origin of the coordinate patch to simplify our computations. Recall equation \eqref{adapted}, at the origin of the \kahler normal coordinate, we have $\frac{\partial \phi}{\partial z}(q)=0$, and hence locally,

 \begin{equation}\label{phidd}\Psi_n^q(w)=\lambda_q \overline{\frac{\partial\Pi_n(z,w)}{\partial z}|}_{z=q}.\end{equation}
 We can further rewrite $\Psi_n^q(w)$ as follows: choose any orthonormal basis  $\{\psi_{n,1},\cdots,\psi_{n,d_n}\}$ of $\hn$ with respect to the inner product \eqref{innera} (or \eqref{innera2}), then the Bergman kernel is $$\Pi_n(z,w)=\sum _{j=1}^{d_n}\psi_{n,j}(z)\overline{\psi_{n,j}(w)},$$thus,
$$\Psi_n^q(w)=\lambda_q \sum _{j=1}^{d_n}\overline{\frac{\partial\psi_{n,j}}{\partial z}(q)}\psi_{n,j}(w).$$

 Note that
the $L^2$-norm of $\Psi_n^q(w)$ is $1$ by assumption,  hence, %by orthonormality of $\{\psi_{n,j}\}$,

$$\begin{aligned} 1&=\|\Psi_n^q(w)\|_{h^n}^2=\|\lambda_q \sum _{j=1}^{d_n}\overline{\frac{\partial\psi_{n,j}}{\partial z}(q)}\psi_{n,j}(w)\|^2_{h^n}\\ &=|\lambda_q|^2\sum_{j=1}^{d_n}\frac{\partial\psi_{n,j}}{\partial z}(q)\overline{\frac{\partial\psi_{n,j}}{\partial z}(q)}=|\lambda_q|^2 (\partial_z \partial
_{\bar z}\Pi_n)(q,q)\end{aligned}.$$
Hence, combining the expression of $|\lambda_q|^2$ with \eqref{phidd}, we have,
$$
|\Psi_n^q(w)|^2=\frac{|\partial_z\Pi_n(z,w)|_{z=q}|^2}{(\partial_z\partial
_{\bar z}\Pi_n)(q,q)}.
$$
Note that $\overline{\Pi_n(z,w)}=\Pi_n(w,z)$, thus $|\partial_z\Pi_n(z,w)|_{z=q}|^2=|  \partial_{\bar z}\overline{\Pi_n(z,w)}|_{z=q}|^2=|  \partial_{\bar z}\Pi_n(w,z)|_{z=q}|^2=|\partial_{\bar w}\Pi_n(z,w)|_{w=q}|^2$, thus,
\begin{equation}\label{ddk}
|\Psi_n^q(w)|^2=\frac{|\partial_{\bar w}\Pi_n(z,w)|_{w=q}|^2}{(\partial_z\partial
_{\bar z}\Pi_n)(q,q)|}.
\end{equation}

Now we complete the proof of Theorem \ref{thm2} if we combine \eqref{ddddd}\eqref{ddk}.

\section{Proof of Theorem \ref{main2}}
\subsection{Gaussian random $SU(2)$ polynomials}
In this subsection, let's compute the rescaling limit $D_\infty(z|0)$ for Gaussian random $SU(2)$ polynomials.

We choose the affine coordinate at $q=0$. % which is an adapted frame (recall \eqref{fsadapt}).
 As in \S \ref{su2}, we still use the normalized Bergman kernel
   $$\Pi_n(z,w)=(1+z\bar w)^n.$$
Thus we have,

$$(\d_{\bar w}\Pi_n)(z,0)=nz,\,\,\,\, (\d_z \d_{\bar z} \Pi_n)(0,0)=n.$$
Thus we have the following exact formula for the $SU(2)$ polynomials,

$$D^{SU(2)}_n(z|0)=\frac{\sqrt{-1}}{2\pi}\partial\bar\partial \log \left((1+|z|^2)^{n}-n|z|^2 \right).$$

We expand the right hand side to get,
$$D^{SU(2)}_n(z|0)=\frac{1}{\pi}[\frac {n\left((1+|z|^2)^{n-1}-1+(n-1)(1+|z|^2)^{n-2}|z|^2\right)}{(1+|z|^2)^n-n|z|^2}$$$$-\frac{n^2((1+|z|^2)^{n-1}-1)^2|z|^2}{((1+|z|^2)^n-n|z|^2)^2 }] \frac{\sqrt{-1}}{2}dz\wedge d\bar z.$$

Now we rescale $z\to \frac{z}{\sqrt n}$ to get the limit,
  \begin{align*}D_\infty^{SU(2)}(z|0):&=\lim_{n\to\infty} D_{n}^{SU(2)}(q+\frac{z}{\sqrt n}|q)\\&=\frac{1}{\pi}\left[\frac {\left(e^{|z|^2}-1+e^{|z|^2}|z|^2\right)}{e^{|z|^2}-|z|^2}-\frac{(e^{|z|^2}-1)^2|z|^2}{(e^{|z|^2}-|z|^2)^2}\right] \frac{\sqrt{-1}}{2}dz\wedge d\bar z, \end{align*}

which can be rewritten as

 $$D_\infty^{SU(2)}(z|0)=\frac{\sqrt{-1}}{2 \pi}\partial\bar\partial \log \left(e^{|z|^2}-|z|^2\right).$$
This proves Theorem \ref{main2} for Gaussian random $SU(2)$ polynomials.

\subsection{Proof of Theorem \ref{main2}}
Let's turn to prove Theorem \ref{main2} for the general cases.

 We have to apply the similar estimates about the Bergman kernel as in \S \ref{profooo}. We continue to use the \kahler normal coordinate with the origin at $q=0$ as in \S \ref{proofoooo}.

By Theorem \ref{thm2}, the rescaling limit of $D_n(z|q)$ is given as

$$D_{n}(\frac{v}{\sqrt n}|0)=\frac{\sqrt{-1}}{2 \pi}\partial\bar \partial \log \left|\Pi_n(\frac v{\sqrt n},\frac v{\sqrt n})-\frac{|(\partial_{\bar w}\Pi_n)(\frac v{\sqrt n},0)|^2}{(\partial_z\partial
_{\bar z}\Pi_n)(0,0)}\right |.$$

 Theorem \ref{main2} follows once we find the rescaling limits of $\Pi_n(\frac v{\sqrt n},\frac v{\sqrt n})$ and $(\partial_{\bar w}\Pi_n)(\frac v{\sqrt n},0)$ and the estimate of $(\partial_z\partial
_{\bar z}\Pi_n)(0,0)$.

Let's recall the rescaling limit of the Bergman kernel off the diagonal,
$$ \Pi_n(q+\frac{v}{\sqrt n}, q+\frac{u}{\sqrt n})=n(e^{v \cdot {\bar u}}+\frac{1}{\sqrt n}p_1+...).$$

%First recall the expansion \eqref{kahler}, around the small neighborhood of the origin, the local \kahler potential satisfies $$n\phi(\frac v {\sqrt n})\sim |v|^2.$$
%If we combine this with the full expansion of the Bergman kernel
%$$\Pi_n(z,z)=ne^{n\phi}(1+a_1(z)n^{-1}+a_2(z)n^{-2}+\cdots),$$
  As in \S\ref{profooo}, we first have the following asymptotics at $q=0$,\begin{equation}\label{eq1}\Pi_n(\frac v{\sqrt n},\frac v{\sqrt n})\sim ne^{|v|^2},\,\,\,(\partial_{\bar w}\Pi_n)(\frac v{\sqrt n},0)\sim n^{\frac 32}v. \end{equation}

% we have the estimate
%\begin{equation}\label{eq3}\end{equation}
Let's recall the $C^\infty$-expansion of the Bergman kernel on the diagonal,
$$
\Pi_n(z,z)=ne^{n\phi}(1+a_1(z)n^{-1}+a_2(z)n^{-2}+\cdots).
$$
We take $\d\dbar$ on both sides to get the full expansion,
$$(\partial_z\partial
_{\bar z}\Pi_n)(z,z)=n^3e^{n\phi}|\d \phi|^2(1+a_1n^{-1}+\cdots)+n^2e^{n\phi}\d\dbar \phi(1+a_1n^{-1}+\cdots)$$
$$+2n^2e^{n\phi}\Re(\d \phi(\d\bar a_1n^{-1}+\cdots))+ne^{n\phi}(\d\dbar a_1n^{-1}+\cdots).$$

Using identities \eqref{adapted} at the origin of the \kahler normal coordinate, we have,
\begin{equation}\label{eq2}(\partial_z\partial
_{\bar z}\Pi_n)(0,0)=n^2+na_1+(\d\dbar a_1+a_2)+O(n^{-1}).\end{equation}
 If we combine the asymptotics \eqref{eq1}\eqref{eq2}, we have the universal limit
 $$D_\infty(v|0):=\lim_{n\to\infty} D_{n}(\frac{v}{\sqrt n}|0)=\frac{\sqrt{-1}}{2 \pi}\partial\bar\partial \log \left(e^{|v|^2}-|v|^2\right),$$
 which completes the proof of Theorem \ref{main2}.

%By taking derivatives, we have full expansions,
%$$\d_z \Pi_n(z,z)=n^2\d_z\phi e^{n\phi}(1+n^{-1}a_1+\cdots)+e^{n\phi}(\d_za_1+n^{-1}\d_za_2+\cdots)$$
%and
%$$\d_z \bar\d_z\Pi_n(z,z)=n^2\d_z\phi e^{n\phi}(1+n^{-1}a_1+\cdots)+e^{n\phi}(\d_za_1+n^{-1}\d_za_2+\cdots)$$

\end{document}